\documentclass{amsart}

\usepackage{amsmath}
\usepackage{amsfonts}
\usepackage{amssymb}
\usepackage{graphicx}
\usepackage{hyperref}
\usepackage{mathrsfs}
\usepackage{pb-diagram}
\usepackage{epstopdf}
\usepackage{amsmath,amsfonts,amsthm,enumerate,amscd,latexsym,curves}
\usepackage{bbm}
\usepackage[mathscr]{eucal}
\usepackage{epsfig,epsf,xypic,epic}
\usepackage{caption}
\usepackage{subcaption}

\newtheorem{theorem}{Theorem}[section]

\newtheorem{lemma}[theorem]{Lemma}
\newtheorem{corollary}[theorem]{Corollary}

\newtheorem{definition}[theorem]{Definition}

\newtheorem{example}[theorem]{Example}

\newtheorem{proposition}[theorem]{Proposition}

\newtheorem{remark}[theorem]{Remark}

\newcommand{\dbar}{\bar{\partial}}
\newcommand{\dpsum}[2]{\displaystyle{\sum_{#1}^{#2}}}
\newcommand{\pd}[2]{\frac{\partial #1}{\partial #2}}

\newcommand{\bb}[1]{\mathbb{#1}}
\newcommand{\cu}[1]{\mathcal{#1}}
\newcommand{\til}[1]{\widetilde{#1}}

\def\bC{\mathbb{C}}

\begin{document}

\title[On the Jumping Phenomenon of $\dim_{\mathbb{C}}H^q(\cu{X}_t,\cu{E}_t)$]{On the Jumping Phenomenon of $\dim_{\mathbb{C}}H^q(\cu{X}_t,\cu{E}_t)$}
\author[K. Chan]{Kwokwai Chan}
\address{Department of Mathematics\\ The Chinese University of Hong Kong\\ Shatin\\ Hong Kong}
\email{kwchan@math.cuhk.edu.hk}
\author[Y.-H. Suen]{Yat-Hin Suen}
\address{Department of Mathematics\\ The Chinese University of Hong Kong\\ Shatin\\ Hong Kong}
\email{yhsuen@math.cuhk.edu.hk}

\date{\today}

\begin{abstract}
Let $X$ be a compact complex manifold and $E$ be a holomorphic vector bundle on $X$. Given a deformation $(\cu{X},\cu{E})$ of the pair $(X,E)$ over a small polydisk $B$ centered at the origin, we study the jumping phenomenon of the cohomology groups $\dim_{\mathbb{C}}H^q(\cu{X}_t,\cu{E}_t)$ near $t = 0$. Generalizing previous results of X. Ye \cite{Ye_jumping, Ye_jumping_tangent} (for the tangent bundle $E = T_{\cu{X}_t}$ and exterior powers of the cotangent bundle $E = \Omega^p_{\cu{X}_t}$), we show that there are precisely two cohomological obstructions to the stability of $\dim_{\mathbb{C}}H^q(\cu{X}_t,\cu{E}_t)$, which can be expressed explicitly in terms of the Maurer-Cartan element associated to the deformation $(\cu{X},\cu{E})$. As an application, we study the jumping phenomenon of the dimension of the cohomology group $H^1(\cu{X}_t,\text{End}(T_{\cu{X}_t}))$, which is related to a question raised by physicists \cite{Huybrechts95}.
\end{abstract}

\maketitle

\tableofcontents

\section{Introduction}

Let $X$ be a compact complex manifold and $\pi:\mathcal{X}\rightarrow B$ be a small deformation of $X=\pi^{-1}(0)$ over a small polydisk $B$ centered at the origin in some complex vector space. Suppose that $\mathcal{E}$ is a coherent sheaf on $\mathcal{X}$ which is flat over $B$. Then $(\mathcal{X},\mathcal{E})$ is a deformation of the pair $(X, \mathcal{E}|_{X})$.

It is known by Grauert's direct image theorem that
the dimension $\dim_{\bC}H^q(\mathcal{X}_t,\mathcal{E}_t)$ is an upper semi-continuous function in $t \in B$. Moreover, we have the following characterization for when the dimension $\dim_{\bC}H^q(\mathcal{X}_t,\mathcal{E}_t)$ is locally constant, also due to Grauert.

\begin{theorem}[Grauert \cite{Grauert_DIT}]\label{prop:locally_freeness}
Let $\pi:\mathcal{X}\rightarrow B$ be a flat proper holomorphic map between complex analytic spaces $\mathcal{X},B$ with $B$ being reduced and connected. Suppose that $\mathcal{E}$ is a coherent sheaf on $\mathcal{X}$ that is flat over $B$. Let $k(t):=\mathcal{O}_{B,t}/\mathfrak{m}_t$ be the residue field at $t\in B$ and $\mathcal{E}_t$ be the pullback of $\mathcal{E}$ to $\mathcal{X}_t$. Then the following are equivalent:
\begin{itemize}
\item [(a)] The function
$$t\mapsto\dim_{\bC}H^q(\mathcal{X}_t,\mathcal{E}_t)$$
is locally constant in $t\in B$.

\item [(b)] The sheaf $R^q\pi_*\mathcal{E}$ is locally free and the natural map
$$R^q\pi_*\mathcal{E}\otimes k(t)\rightarrow H^q(\mathcal{X}_t,\mathcal{E}_t)$$
is an isomorphism.
\end{itemize}
\end{theorem}

However, condition (b) in the theorem above is not easy to check in general even when $\mathcal{E}$ is locally free. In \cite{Ye_jumping, Ye_jumping_tangent}, X. Ye studied the jumping phenomenon of the dimensions $\dim_{\bC}H^q(X,\bullet)$ under small deformations of $X$, where $\bullet=\Omega^p_X,T_X$. He found two explicit obstructions $O^q_{n,n-1}$, $O^{q-1}_{n,n-1}$ and proved that the dimension of $H^q(X,\bullet)$ does not jump if and only if $O^q_{n,n-1}\equiv 0$ and $O^{q-1}_{m,m-1}\equiv 0$ for all $n,m\geq 1$.

In this paper, we generalize Ye's results to a much more general setting, namely, when $X$ is a compact complex manifold and $E$ is an arbitrary holomorphic vector bundle on $X$. Let $(\mathcal{X},\mathcal{E})$ be a small deformation of $(X,E)$ over a polydisk $B$ centered at the origin in some finite dimensional complex vector space.
We assume that $\mathcal{E}$ is flat over $B$ via the proper holomorphic submersion $\pi:\mathcal{X}\rightarrow B$. Let $\mathcal{X}_t := \pi^{-1}(t)$ and $\mathcal{E}_t := \mathcal{E}|_{\mathcal{X}_t}$. We are interested in characterizing when the dimension $\dim_{\bC}H^q(\mathcal{X}_t,\mathcal{E}_t)$ stays constant near $t = 0$.

Following \cite{Ye_jumping, Ye_jumping_tangent}, we formulate the jumping phenomenon of $\dim_{\bC}H^q(\mathcal{X}_t,\mathcal{E}_t)$ as an extension problem, namely, whether we can extend a nonzero element in $H^q(X,E)$ to one in a nearby fiber $H^q(\mathcal{X}_t,\mathcal{E}_t)$. In general, such extensions may not exist and it suffices to find obstructions to this extension problem. We will see that Ye's explicit formulae for the obstructions can be generalized to our general setting.
On the other hand, while Ye applied a version of Grauert's direct image theorem, which states that $R^q\pi_*\mathcal{E}$ is a quotient of two locally free sheaves of finite ranks over $B$, thereby allowing him to apply an algebraic approach,
here we adapt a differential-geometric approach, following \cite{Huang95, CS_pair}.

We will formulate the problem directly as extending $E$-valued differential forms over $B$, which means that, in contrast to \cite{Ye_jumping, Ye_jumping_tangent}, we are going to work with sheaves of {\em infinite rank}. A key step is to obtain an explicit description of $R^q\pi_*\mathcal{E}$, using an acyclic resolution $(\mathcal{D}^{\bullet},\bar{D}^{\bullet})$ of the sheaf $\mathcal{E}$ constructed from the differential operators $\bar{D}^\bullet$ studied in \cite{Huang95, CS_pair} (see Section \ref{sec:acyclic_resol}). The operators $\bar{D}^{\bullet}$ capture the holomorphic structures of the deformed pairs $\{(\mathcal{X}_t,\mathcal{E}_t)\}_{t\in B}$ (see \cite{CS_pair} or Section \ref{sec:def_pair} in this paper).
It turns out that essentially the same strategy as in Ye's proofs works.
An advantage of our geometric approach is that the computation of the obstructions becomes much neater and more transparent, as compared to the \v{C}ech calculations in \cite{Ye_jumping, Ye_jumping_tangent}.
Our main result is as follows (see Section \ref{sec:obstr}, in particular, Theorem \ref{thm:main_2} and Equations \eqref{eqn:obstruction} $\&$ \eqref{eqn:obstruction2} for the details):

\begin{theorem}\label{thm:main}
Let $\{(A(t),\varphi(t))\}_{t\in B}$ be the family of Maurer-Cartan elements associated to the small deformation $(\mathcal{X},\mathcal{E})$ of $(X,E)$. We define the {\em $n$-th order obstruction maps} $O^{i}_{n,n-1}:H^i((\pi_*\mathcal{D}^{\bullet})_0\otimes\mathcal{O}_{B,0}/m_0^n)\rightarrow H^{i+1}((\pi_*\mathcal{D}^{\bullet})_0\otimes\mathcal{O}_{B,0}/m_0^n)$, where $i=q, q - 1$, by
$$O^{i}_{n,n-1}\left([\alpha_{n-1}]\right)=\left[t^{n-1}\sum_{j=0}^{n-1}(\varphi^{n-j}\lrcorner\nabla+A^{n-j})\alpha^j_{n-1}\right].$$
Then the function $t\mapsto\dim_{\mathbb{C}}H^q(\mathcal{X}_t,\mathcal{E}_t)$ is locally constant if and only if $O_{m,m-1}^q\equiv 0$ and $O^{q-1}_{n,n-1}\equiv 0$ for all $m,n\geq 1$.
\end{theorem}

We apply this theorem to study the jumping phenomenon of the dimension $\dim_{\bb{C}}H^1(\cu{X}_t,\text{End}(T_{\cu{X}_t}))$, which is related to a question raised by physicists \cite{Huybrechts95}. It is conjectured that $\dim_\mathbb{C} H^1(\cu{X}_t,\text{End}(T_{\cu{X}_t}))$ does not jump along {\em any} deformation of a {\em Calabi-Yau} manifold $X$.
What we obtain is the following weaker statement (see Section \ref{sec:CY_stable}):

\begin{theorem}(=Theorem \ref{thm:CY_jumping})\label{thm:application}
Suppose that $X$ is a Calabi-Yau manifold such that the deformation of the pair $(X, T_X)$ is unobstructed, then $\dim_{\mathbb{C}}H^1(\cu{X}_t,\text{End}(T_{\cu{X}_t}))$ does not jump at $t=0$ for any small deformation of $X$ .
\end{theorem}

\section*{Acknowledgment}
We are grateful to Xuanming Ye and Robert Lazarsfeld for various useful discussions via emails. We would also like to thank Conan Leung for encouragement and some useful suggestions, and also the referees for valuable comments. The work of the first named author described in this paper was substantially supported by grants from the Research Grants Council of the Hong Kong Special Administrative Region, China (Project No. CUHK404412 $\&$ CUHK400213).

\section{Deformations of Pairs}\label{sec:def_pair}

In this section, we briefly review the deformation theory of a pair $(X,E)$, where $X$ is a compact complex manifold and $E$ is a holomorphic vector bundle on $X$, following the exposition in \cite{CS_pair} (cf. \cite{Huang95}), and recall several useful facts.

\begin{definition}
Let $B$ be a small polydisk in some finite dimensional complex vector space containing the origin. A {\em deformation} of $(X,E)$ over $B$ consists of a surjective proper holomorphic submersion $\pi:\mathcal{X}\rightarrow B$ from a complex manifold $\mathcal{X}$ to $B$, together with a holomorphic vector bundle $\mathcal{E}$ on $\mathcal{X}$, such that $\pi^{-1}(0)=X$ and $\mathcal{E}|_{\pi^{-1}(0)}=E$.
\end{definition}

Given such a deformation of $(X,E)$, we put $\mathcal{X}_t:=\pi^{-1}(t)$ and $\mathcal{E}_t:=\mathcal{E}|_{\mathcal{X}_t}$. Since $B$ is contractible, a theorem of Ehresmann
implies that we can choose a diffeomorphism $F:\cu{X}\to X\times B$ and a bundle isomorphism $F':\cu{E}\to E\times B$ covering $F$ such that $F,F'$ are holomorphic with respect to $t$.
Notice that there are two complex structures on $X\times B$: one comes from the push-forward of the complex structure on $\cu{X}$ and the other comes from the product structure on $X\times B$; we denote these complex structures by $\cu{J}$ and $\cu{J}_0$, respectively.

Let $\varphi(t)\in\Omega^{0,1}(T_X)$ be the family of Maurer-Cartan elements which corresponds to the family $\mathcal{X}\rightarrow B$. In \cite{CS_pair}, we considered a holomorphic family of differential operators $\bar{D}_t^q:\Omega^{0,q}(E)\rightarrow\Omega^{0,q+1}(E)$ defined locally by
$$\bar{D}_t^q\left(\sum_j\alpha_j\otimes e_j(t)\right) := \sum_j(\dbar+\varphi(t)\lrcorner\partial)\alpha_j\otimes e_j(t),$$
where $\{e_j(t)\}$ is the push-forward of a local holomorphic frame on $\mathcal{E}_t$ by $F'$.
By choosing a Hermitian metric on $E$, one can express the operator $\bar{D}_t$, in terms of the associated Chern connection $\nabla$, as
$$\bar{D}_t^q = \dbar_E + \varphi(t)\lrcorner\nabla + A(t),$$
for some $A(t)\in\Omega^{0,1}(\text{End}(E))$.

Then (one direction of) Theorem 1.2 in \cite{CS_pair} says that the family of elements $(A(t),\varphi(t))\in\Omega^{0,1}(A(E))$ satisfies the Maurer-Cartan equation
$$\dbar_{A(E)}(A(t),\varphi(t))+\frac{1}{2}[(A(t),\varphi(t)),(A(t),\varphi(t))]=0$$
for $t \in B$; here $A(E)$ is the Atiyah extension of $E$. This in turn is equivalent to the fact that the family of operators $\{\bar{D}_t^q\}$ satisfies the integrability condition:
$$\bar{D}_t^q\bar{D}_t^{q-1}=0.$$

Another important feature of the operator $\bar{D}_t$, which is going to be useful later, is that its cohomology computes precisely the Dolbeault cohomology of $(\mathcal{X}_t,\mathcal{E}_t)$:
\begin{proposition}[\cite{CS_pair}, Proposition 3.13]\label{lem:unnatural_isom}
For each fixed $t\in B$, we have
$$H^q(\mathcal{X}_t,\mathcal{E}_t)\cong H^q((\pi_*\mathcal{D}^{\bullet})_t\otimes k(t))\cong H^q\left(\Omega^{0,\bullet}(E),\bar{D}_t\right),$$
for any $q\geq 0$.
\end{proposition}

\section{An acyclic resolution for $\mathcal{E}$}\label{sec:acyclic_resol}

From this point on, for the purpose of simplifying computations and formulae, we will assume that the base $B$ of the deformation is of complex dimension one. We also abuse notations by writing $\cu{X}$ for the complex manifold $(X\times B,\cu{J})$ and $\cu{E}$ for the vector bundle $E\times B$ equipped with the holomorphic structure induced from $\cu{E}$ via pushing forward by $F':\cu{E}\to E\times B$.

In this section, we will construct an acyclic resolution of the sheaf $\mathcal{E}$ in order to get an explicit description of the direct image sheaf $R^q\pi_*\mathcal{E}$.

To begin with, we define an operator $\dbar_{\cu{E},B}:\Omega_{\cu{J}_0}^{0,q}(\cu{E})\to\Omega_{\cu{J}_0}^{0,q+1}(\cu{E})$ by
$$\dbar_{\cu{E},B}\bigg(\sum_{j}s_je_j(t)\bigg):=\sum_j\dbar_Bs_j\otimes e_j(t).$$
Here $\Omega_{\cu{J}_0}^{0,\bullet}$ is the space of smooth $(0,\bullet)$-forms on $X\times B$ with respect to the product complex structure $\cu{J}_0$. If we choose a different frame $f_k(t)$, then $e_j(t)=\sum_kg_j^k(t)f_k(t)$ for some local smooth functions $g_j^k$ on $X\times B$ which are holomorphic in $t$. Hence $\dbar_{\cu{E},B}$ is well-defined.

For each $q\geq 0$, we define a sheaf of $\cu{O}_{\cu{X}}$-modules $\cu{D}^q$, which plays a key role throughout this paper: Let $\pi_X:\cu{X}\to X$ be the projection onto $X$ (which is not necessarily holomorphic). The sheaf $\cu{U}\mapsto\Omega^{0,k}_{\cu{J}_0}(\cu{U},\cu{E})$ has a typical direct summand given by
$$\til{\cu{D}}^{q,p}:=\pi^*\Omega^{0,p}_B\otimes\pi_X^*\Omega^{0,q}\otimes\cu{E},\quad p+q=k,$$
which carries an $\cu{O}_{\cu{X}}$-module structure via multiplication by $\cu{J}$-holomorphic functions. The operator $\dbar_{\cu{E},B}$ acts on $\bigoplus_{p\geq 0}\til{\cu{D}}^{q,p}$, so we obtain a complex $(\widetilde{\mathcal{D}}^{q,\bullet},\dbar_{\cu{E},B}^{\bullet})$ for each $q$. We then define the sheaf of $\cu{O}_{\cu{X}}$-modules $\cu{D}^q$ by
$$\cu{D}^q:\cu{U}\mapsto\{s\in\Gamma_{smooth}(\cu{U},\til{\cu{D}}^{q,0}) \mid \dbar_{\cu{E},B}s=0\}.$$

Clearly, $\mathcal{D}^{\bullet}\subset\widetilde{\mathcal{D}}^{\bullet,0}$ as $\mathcal{O}_{\mathcal{X}}$-submodules. Since $\bar{D}_t$ varies holomorphically in the variable $t$, it induces a sheaf map $\bar{D}^q:\mathcal{D}^q\rightarrow\mathcal{D}^{q+1}$ for each $q\geq 0$. Moreover, since the kernel of $\bar{D}_t:\Omega^0(E)\to\Omega^{0,1}(E)$ is precisely the space of holomorphic sections of $\cu{E}_t$, the sheaf $\cu{E}$, as a sheaf of $\cu{O}_{\cu{X}}$-modules, can be identified with the following sheaf of $\cu{O}_{\cu{X}}$-modules:
$$\cu{U}\mapsto\{s\in\Gamma_{smooth}(\cu{U},\cu{E}) \mid \bar{D}s=\dbar_{\cu{E},B}s=0\}.$$
The push-forwards of $\til{\cu{D}}^{q,p}$ and $\cu{D}^q$ by $\pi:\cu{X}\to B$ carry natural $\mathcal{O}_B$-module structures via multiplication.

\begin{lemma}\label{lem:no_higher_cohomology_sheaf}
For each $p,q\geq 0$, the sheaf $\widetilde{\mathcal{D}}^{q,p}$ is fine and the complex $(\pi_*\widetilde{\mathcal{D}}^{q,\bullet},\pi_*\dbar_{\cu{E},B}^{\bullet})$ has no higher cohomology sheaves, i.e.,
$$\mathcal{H}^p(\pi_*\widetilde{\mathcal{D}}^{q,\bullet}) = 0$$
for all $p \geq 1$.
\end{lemma}
\begin{proof}
Fineness is clear because we can apply a partition of unity to conclude that $\widetilde{\mathcal{D}}^{q,p}$ has no higher direct images.

To prove that $(\pi_*\widetilde{\mathcal{D}}^{q,\bullet},\pi_*\dbar_{\cu{E},B}^{\bullet})$ has no higher cohomology, we recall that $\mathcal{H}^p(\pi_*\widetilde{\mathcal{D}}^{q,\bullet})$ is the sheafification of
$$W\mapsto H^p(\Gamma(\pi^{-1}(W),\widetilde{\mathcal{D}}^{q,\bullet})).$$
It suffices to prove that $H^p(\Gamma(\pi^{-1}(W),\widetilde{\mathcal{D}}^{q,\bullet}))=0$ for any polydisk $W\subset B$ and all $p\geq 1$. Let $\alpha\in\Gamma(\pi^{-1}(W),\widetilde{\mathcal{D}}^{q,p})$ and $\{U_i\}$ be a locally finite open covering of $X\subset\pi^{-1}(W)$ by coordinates charts. Let $\alpha_i$ be the restriction of $\alpha$ on $U_i\times W$. Write
$$\alpha_i=\sum_{I,J}\alpha_{IJ,i}(z,\bar{z},t,\bar{t})d\bar{t}^J\otimes d\bar{z}^I=:\sum_{I}\alpha_{I,i}\otimes d\bar{z}^I.$$
Then $\dbar_{\cu{E},B}\alpha=0$ simply means that, for each $I$,
$$0 = \dbar_{\cu{E},B}\left(\sum_{J}\alpha_{IJ,i}d\bar{t}^J\right) = \dbar_B\alpha_{I,i}.$$

Hence, for fixed $z$, we can apply the Dolbeault lemma on $W$ to conclude that
$$\alpha_{I,i}=\dbar_B\beta_{I,i},$$
for some $\beta_{I,i}\in\Omega^{0,p-1}_B(W)$. Since $\alpha_{I,i}$ varies smoothly in $z$ and $\bar{z}$, we see from the proof of the Dolbeault-Grothendieck lemma that $\beta_{I,i}$ can be chosen to be smooth in $z$ as well. Let $\{\psi_i\}$ be a partition of unity on $X$ subordinate to the covering $\{U_i\}$. Define
$$\beta:=\sum_{I,i}\psi_i\beta_{I,i}\otimes d\bar{z}^I.$$
Then $\beta\in\Gamma(\pi^{-1}(W),\widetilde{\mathcal{D}}^{q,p-1})$ and
$$\dbar_{\cu{E},B}\beta = \sum_{I,i}\psi_i(\dbar_B\beta_{I,j})\otimes d\bar{z}^I = \sum_{I,i}\psi_i\alpha_{I,i}\otimes d\bar{z}^I = \left(\sum_i\psi_i\right)\alpha=\alpha.$$
We have the first equality simply because $\{\psi_i\}$ are all independent of $t$ and $\bar{t}$, and the second last equality follows from the fact that $\alpha$ is a global section on $\pi^{-1}(W)$.
\end{proof}

\begin{lemma}\label{lem:acyclic}
For each $q\geq 0$, the sheaf $\mathcal{D}^q$ is acyclic with respect to the left-exact functor $\pi_*$.
\end{lemma}
\begin{proof}
Lemma \ref{lem:no_higher_cohomology_sheaf} shows that $(\widetilde{\mathcal{D}}^{q,\bullet},\dbar_{\cu{E},B}^{\bullet})$ is a fine resolution of $\mathcal{D}^q$ and so $R^q\pi_*\mathcal{D}^q\cong\mathcal{H}^p(\pi_*\widetilde{\mathcal{D}}^{q,\bullet})=0$ for all $q\geq 1$.
\end{proof}

\begin{proposition}\label{lem:acyclic_resolution}
The complex of sheaves $(\mathcal{D}^{\bullet},\bar{D}^{\bullet})$ is an acyclic resolution of $\mathcal{E}$ with respect to the left-exact functor $\pi_*$. In particular, we have
$$R^q\pi_*\mathcal{E}\cong \mathcal{H}^q(\pi_*\mathcal{D}^{\bullet})$$
as $\mathcal{O}_B$-modules.
\end{proposition}
\begin{proof}
By Lemma \ref{lem:acyclic}, $R^p\pi_*\mathcal{D}^q=0$ for all $p\geq 1$. It remains to prove that it defines a resolution of $\mathcal{E}$. We need to show that for any point $(x,t)\in\mathcal{X}=X\times B$, the sequence of stalks
$$0\rightarrow\mathcal{E}_{(x,t)}\rightarrow\mathcal{D}^0_{(x,t)}\rightarrow\mathcal{D}^1_{(x,t)}\rightarrow\cdots$$
is exact. The exactness of
$$0\rightarrow\mathcal{E}_{(x,t)}\rightarrow\mathcal{D}^0_{(x,t)}\rightarrow\mathcal{D}^1_{(x,t)}$$
follows from the fact that $\bar{D}^0$ and $\dbar_{\mathcal{E}}$ share the same kernel. For the remaining exactness, we will focus on the case $t=0$; the same argument works for general $t\in B$.

Recall that $\bar{D}_t$ is locally defined by
$$\bar{D}_t\left(\sum_j\alpha_j\otimes e_j(t)\right) := \sum_j(\dbar+\varphi(t)\lrcorner\partial)\alpha_j\otimes e_j(t),$$
so it suffices to prove the exactness for the case $\mathcal{E}=\mathcal{O}_{\mathcal{X}}$.

We would like to first work over $\mathbb{C}[[t]]$ instead of $\mathbb{C}\{t\}$ (where the latter is the ring of convergent power series). Let $U\subset X$ be a polydisk, and denote
\begin{align*}
\Omega^{0,\bullet}(U)\{t\} & := \Omega^{0,\bullet}(U)\otimes_{\mathbb{C}}\mathbb{C}\{t\},\\
\Omega^{0,\bullet}(U)[[t]] & :=\Omega^{0,\bullet}(U)\otimes_{\mathbb{C}}\mathbb{C}[[t]] = \Omega^{0,\bullet}(U)\{t\}\otimes_{\mathbb{C}\{t\}}\mathbb{C}[[t]].
\end{align*}
The Maurer-Cartan element $\varphi(t)$ is gauge equivalent to $0$ on $U$. Hence
$$\dbar+\varphi(t)\lrcorner\partial=e^{v(t)}\dbar e^{-v(t)}$$
for some $v(t)\in\Omega^0(T_U)[[t]]$, where $e^{v(t)}$ acts on $\Omega^{0,q}(U)[[t]]$ by
$$e^{v(t)}\alpha(t)=\sum_{n=0}^{\infty}\frac{(v(t)\lrcorner\partial)^n}{n!}\alpha(t).$$
We can then apply the Dolbeault-Grothendieck lemma with analytic parameter (the $t$-variable) to conclude that $(\Omega^{0,\bullet}(U)[[t]],\bar{D}^{\bullet}_t)$ is an exact complex.

Now, as $\mathbb{C}[[t]]$ is a flat-$\mathbb{C}\{t\}$ module (because $\mathbb{C}[[t]]$ is torsion free and $\mathbb{C}\{t\}$ is a PID), we have
$$H^q(\Omega^{0,\bullet}(U)[[t]])=H^q(\Omega^{0,\bullet}(U)\{t\}\otimes\mathbb{C}[[t]])\cong H^q(\Omega^{0,\bullet}(U)\{t\})\otimes\mathbb{C}[[t]].$$
But we have shown that $H^q(\Omega^{0,\bullet}(U)[[t]])=0$. Therefore, $H^q(\Omega^{0,\bullet}(U)\{t\})\otimes\mathbb{C}[[t]]=0$. If we can show that $H^q(\Omega^{0,\bullet}(U)\{t\})$ is torsion free, we see that $H^q(\Omega^{0,\bullet}(U)\{t\})$ vanishes. Assuming this, we conclude that every $\bar{D}_t$-closed $(0,q)$-form valued power series on $U$ is locally exact.

Now, for any $\bar{D}^q_{(x,0)}$-closed element $\alpha\in\mathcal{D}^q_{(x,0)}$, we can represent it by a $\bar{D}_t$-closed element $\alpha(t)\in\Omega^{0,q}(U)\{t\}$, for some polydisk $U\subset X$. The vanishing of $H^q(\Omega^{0,\bullet}(U)\{t\})$ shows that $\alpha(t)=\bar{D}_t\beta(t)$ for some $\beta(t)\in\Omega^{0,q-1}(U)\{t\}$. This $\beta(t)$ defines an element $\beta\in\mathcal{D}^{q-1}_{(x,0)}$ such that $\bar{D}^{q-1}\beta=\alpha$. This proves the exactness of the complex $(\mathcal{D}^{\bullet},\bar{D}^{\bullet})$.

To complete the proof of the proposition, we need to prove that $H^q(\Omega^{0,\bullet}(U)\{t\})$ is a torsion free $\mathbb{C}\{t\}$-module for $q>1$. In other words, we need to show that if $[\alpha(t)]\in H^q(\Omega^{0,\bullet}(U)\{t\})$ is a nonzero element, then $f(t)\cdot[\alpha(t)]$ is nonzero for all $f(t)\in\mathbb{C}\{t\}-\{0\}$. Since $f(t)$ is invertible if $f(0)\neq 0$, we may assume $f(t)\in (t^N)$ for some $N\geq 1$. We may assume $N$ is chosen such that $f(t)=t^Ng(t)$ with $g(0)\neq 0$. Again, we can invert $g(t)$, so we can further assume $f(t)=t^N$. Then the vanishing of $f(t)\cdot[\alpha(t)]=[f(t)\cdot\alpha(t)]$ means
$$t^N\alpha(t)=\bar{D}_t\beta(t)=(\dbar+\varphi(t)\lrcorner\partial)\beta(t),$$
for some $\beta(t)\in\Omega^{0,q-1}(U)\{t\}$. Since both $\alpha(t)$ and $\beta(t)$ are holomorphic in $t$, the equation shows that $\beta(t)$ is in fact $\bar{D}_t$-closed up to order $N-1$.

We first prove the following
\begin{lemma}\label{lem:dbar}
For any $\dbar$-closed $\beta\in\Omega^{0,q-1}(U)$, $q> 1$, there exists $\beta(t)\in\Omega^{0,q-1}(U)\{t\}$ such that
$$\beta(0)=\beta\text{  and  }\bar{D}_t\beta(t)=0.$$
\end{lemma}
\begin{proof}[Proof of Lemma~\ref{lem:dbar}]
Since $\beta$ is $\dbar$-closed on the polydisk $U$, it must be $\dbar$-exact. Write $\beta=\dbar\alpha$ for some $\alpha\in\Omega^{0,q-2}(U)$. Define
$$\beta(t):=\beta+\varphi(t)\lrcorner\partial\alpha\in\Omega^{0,q-1}(U)\{t\}.$$
Then $\beta(0)=\beta$. Since $\bar{D}_t^2=0$, we have
\begin{align*}
\bar{D}_t\beta(t) & = \dbar\varphi(t)\lrcorner\partial\alpha+\varphi(t)\lrcorner\dbar\partial\alpha+\varphi(t)\lrcorner\partial\beta+\frac{1}{2}[\varphi(t),\varphi(t)]\lrcorner\partial\alpha\\
& = \left(\dbar\alpha+\frac{1}{2}[\varphi(t),\varphi(t)]\right)\lrcorner\partial\alpha+(\varphi(t)\lrcorner\partial\dbar\alpha-\varphi(t)\lrcorner\partial\dbar\alpha)\\
& = 0,
\end{align*}
as desired.
\end{proof}

With this lemma in hand, we see that $\alpha(t)$ is $\bar{D}_t$-exact and this proves that $H^q(\Omega^{0,\bullet}(U)\{t\})$ is torsion free.

Since $\beta_0$ is $\dbar$-closed, we can choose $\beta_1(t)\in\Omega^{0,q-1}(U)\{t\}$ such that
$$\beta_1(0)=\beta_0\text{  and  }\bar{D}_t\beta_1(t)=0.$$
Then we have
$$t^{N-1}\alpha(t) = \bar{D}_t\left(\frac{\beta(t)-\beta_1(t)}{t}\right) = \bar{D}_t\gamma_1(t).$$
If $N=1$, we are done. Otherwise, by evaluating at $t=0$, we see that $\gamma_1(0)$ is $\dbar$-closed. Hence we can find $\beta_2(t)$ such that
$$\beta_2(0)=\gamma_1(0)\text{  and  }\bar{D}_t\beta_2(t)=0.$$
Hence
$$t^{N-2}\alpha(t) = \bar{D}_t\left(\frac{\gamma_1(t)-\beta_2(t)}{t}\right).$$
Repeating this process, we will arrive at the conclusion that
$$\alpha(t)=\bar{D}_t\gamma_N(t)$$
for some $\gamma_N(t)\in\Omega^{0,q-1}(U)\{t\}$. This completes the proof of the proposition.
\end{proof}

\section{Obstructions}\label{sec:obstr}

In this section, we will find out explicitly the obstruction maps for extending a given element of $H^q(X,E)$. In \cite{Ye_jumping, Ye_jumping_tangent}, X. Ye used Grauert's direct image theorem to obtain a complex of locally free $\mathcal{O}_B$-modules of finite ranks to compute the obstruction maps; here we will instead use the infinite-dimensional complex of $\mathcal{O}_B$-modules $(\pi_*\mathcal{D}^{\bullet},\bar{D}^{\bullet})$. We will see that more or less the same strategy of proofs in \cite{Ye_jumping, Ye_jumping_tangent} is going to work in our infinite-dimensional setting as well. We will give most of the details of the proofs in order to make this paper more self-contained.

Recall that Proposition \ref{lem:acyclic_resolution} gives an isomorphism of $\mathcal{O}_B$-modules:
$$R^q\pi_*\mathcal{E}\cong \mathcal{H}^q(\pi_*\mathcal{D}^{\bullet}).$$
Together with Proposition \ref{lem:unnatural_isom}, we see that it is equivalent to work with the sheaf $\mathcal{H}^q(\pi_*\mathcal{D}^{\bullet})$ and the cohomology group $H^q((\pi_*\mathcal{D}^{\bullet})_0\otimes k(0))$.
Tensoring the stalk $(\pi_*\mathcal{D}^{\bullet})_0$ with $\mathcal{O}_{B,0}/m_0^{n+1}$ over $\mathcal{O}_{B,0}$, we obtain a complex
$$((\pi_*\mathcal{D}^{\bullet})_0\otimes\mathcal{O}_{B,0}/m_0^{n+1},\bar{D}_n^\bullet),$$
where $\bar{D}_n^{\bullet}$ is naturally induced from $\bar{D}^{\bullet}$.

Given $\alpha\in\text{ker}(\dbar^q_E)$, and supposing that we have a local extension $\alpha_{n-1}\in\Gamma(U,\pi_*\mathcal{D}^q)$ of $\alpha$ such that
$$j_0^{n-1}(\bar{D}^q\alpha_{n-1})(t)=0,$$
we define the {\em obstruction map}
$$O_{n,n-1}^q:H^q((\pi_*\mathcal{D}^{\bullet})_0\otimes\mathcal{O}_{B,0}/m_0^n)\rightarrow
H^{q+1}((\pi_*\mathcal{D}^{\bullet})_0\otimes\mathcal{O}_{B,0}/m_0^n)$$
by
\begin{equation}\label{eqn:obstruction}
\begin{split}
O_{n,n-1}^q[j_0^{n-1}(\alpha_{n-1})(t)] :=[t^{n-1}\cdot(j_0^n(\bar{D}^q\alpha_{n-1})(t)/t^n)]
\end{split}
\end{equation}

\begin{remark}
The $(n-1)$-st jet can be viewed as an element in $(\pi_*\mathcal{D}^q)_0\otimes\mathcal{O}_{B,0}/m_0^n$. The map $O^q_{n,n-1}$ factors through a map
$$O^q_n:H^q((\pi_*\mathcal{D}^{\bullet})_0\otimes\mathcal{O}_{B,0}/m_0^n)\rightarrow H^{q+1}((\pi_*\mathcal{D}^{\bullet})_0\otimes\mathcal{O}_{B,0}/m_0),$$
given by
$$O^q_n[j_0^{n-1}(\alpha_{n-1})(t)]:=[j^q_0(\bar{D}^q\alpha_{n-1})(t)/t^n].$$
This is well-defined because the cohomology class of $j^q_0(d^q\alpha_{n-1})(t)/t^n$ only depends on the cohomology class of the $(n-1)$-st jet $j_0^{n-1}(\alpha_{n-1})(t)$.

For later use, we also define
$$O_{n,i}^q[j_0^{n-1}(\alpha_{n-1})(t)]:=[t^i\cdot (j^q_0(\bar{D}^q\alpha_{n-1})(t)/t^n)],$$
for $i\geq 0$ and $n\geq 1$
\end{remark}

The following proposition characterizes when an extension exists up to order $n\geq 1$.

\begin{proposition}\label{prop:extension}
For a fixed $n\geq 1$, the following are equivalent:
\begin{itemize}
\item [(1)] For any local section $\alpha_{n-1}$ around $t=0$ such that $j^{n-1}_0(\bar{D}^q\alpha_{n-1})(t)=0$, there exists a local section $\alpha_n$ around $t=0$ such that $j_0^0(\alpha_n-\alpha_{n-1})=0$ and $j_0^n(\bar{D}^q\alpha_n)(t)=0$.
\item [(2)] For any $c_{n-1}\in H^q((\pi_*\mathcal{D}^q)_0\otimes\mathcal{O}_{B,0}/m_0^n)$, there exists $c_n\in H^q((\pi_*\mathcal{D}^q)_0\otimes\mathcal{O}_{B,0}/m_0^{n+1})$ such that $c_n|_{t=0}=c_{n-1}|_{t=0}\in H^q((\pi_*\mathcal{D}^q)_0\otimes k(0))$.
\item [(3)] For any local section $\alpha_{n-1}$ around $t=0$ such that $j_0^{n-1}(\bar{D}^q\alpha_{n-1})(t)=0$, $O_{n,n-1}^q[j^{n-1}_0(\alpha_{n-1})(t)]=0$.
\end{itemize}
\end{proposition}
\begin{proof}
We shall prove that $(1)\Leftrightarrow (2)$ and $(1)\Leftrightarrow (3)$.

For $(1)\Rightarrow (2):$ Let $c_{n-1}\in H^q((\pi_*\mathcal{D}^{\bullet})_0\otimes\mathcal{O}_{B,0}/m_0^n)$ and $\alpha_{n-1}$ be a local section around $t=0$ such that $j_0^{n-1}(\alpha_{n-1})(t)\in\text{ker}(\bar{D}^q_{n-1})$ represents the class $c_{n-1}$. Then $j_0^{n-1}(\bar{D}^q\alpha_{n-1})(t)=0$. By assumption, we can extend $\alpha_{n-1}$ to a local section $\alpha_n$ around $t=0$ such that $j_0^0(\alpha_n-\alpha_{n-1})(t)=0$ and $j_0^n(\bar{D}^q\alpha_n)(t)=0$. Then $\bar{D}^q_n(j_0^n(\alpha_n)(t))=0\in(\pi_*\mathcal{D}^{q+1})_0\otimes\mathcal{O}_{B,0}/m_0^{n+1}$. Set $c_n:=[j_0^n(\alpha_n)(t)]\in H^q((\pi_*\mathcal{D}^{\bullet})_0\otimes\mathcal{O}_{B,0}/m_0^{n+1})$. Since $j_0^0(\alpha_n-\alpha_{n-1})=0$, we have $c_n|_{t=0}=[j_0^0(\alpha_{n-1})(t)]=c_{n-1}|_{t=0}=0$.

For $(2)\Rightarrow (1):$ Let $\alpha_{n-1}$ be such that $j_0^{n-1}(\bar{D}^q\alpha_{n-1})(t)=0$. Extend $c_{n-1}:=[j^{n-1}_0(\alpha_{n-1})(t)]$ to a class $c_n\in H^q((\pi_*\mathcal{D}^{\bullet})_0\otimes\mathcal{O}_{B,0}/m_0^{n+1})$. Let $\alpha_n$ be local section around $t=0$ such that $j_0^n(\alpha_n)(t)$ represents the class $c_n$. Then $j^n_0(\bar{D}^q\alpha_n)(t)=0$. Since $c_n|_{t=0}=c_{n-1}|_{t=0}$, we have
$$j_0^0(\alpha_n-\alpha_{n-1})=\bar{D}_0^{q-1}\gamma,$$
for some $\gamma\in(\pi_*\mathcal{D}^{q-1})_0\otimes k(0)$. Choose any representative $\gamma'$ of $\gamma$ and define
$$\alpha_n':=\alpha_n-\bar{D}^{q-1}\gamma'.$$
Then $j_0^n(\bar{D}^q\alpha_n')(t)=j_0^n(\bar{D}^q\alpha_n)(t)=0$ and $j_0^0(\alpha_n'-\alpha_{n-1})=0$.

For $(1)\Rightarrow (3):$ Let $\gamma:=\alpha_{n-1}-\alpha_n$. Then
$$\bar{D}_n^q\gamma=j^n_0(\bar{D}^q(\alpha_{n-1}-\alpha_n))(t)=t^n\cdot(j^n_0(\bar{D}^q\alpha_{n-1})(t)/t^n),$$
since $j^{n-1}_0(\bar{D}^q\alpha_{n-1})(t)=j_0^n(d^q\alpha_n)(t)=0$. By assumption, $j_0^0(\gamma)(t)=0$, so $\gamma=t\beta$ for some local section $\beta$ around $t=0$. Hence
$$\bar{D}^q_{n-1}j^{n-1}_0(\beta)(t)=t^{n-1}\cdot(j^n_0(\bar{D}^q\alpha_{n-1})(t)/t^n),$$
which means that $O_{n,n-1}^q[j_0^{n-1}(\alpha_{n-1})(t)]=0$.

For $(3)\Rightarrow (1):$ The vanishing of $O_{n,n-1}^q[j_0^{n-1}(\alpha_{n-1})(t)]$ gives an element $\beta\in(\pi_*\mathcal{D}^q)_0\otimes\mathcal{O}_{B,0}/m_0^n$ such that
$$t^{n-1}\cdot(j^n_0(\bar{D}^q\alpha_{n-1})(t)/t^n)=\bar{D}_n^q\beta.$$
Let $\beta'$ be a local section around $t=0$ representing the germ $\beta$ and set $\alpha_n:=\alpha_{n-1}-t\beta'$. Then
\begin{align*}
j^n_0(\bar{D}^q\alpha_n)(t) & = j^n_0(\bar{D}^q\alpha_{n-1})(t)-t\cdot j^{n-1}_0(\bar{D}^q\beta')(t)\\
& = t^n\cdot(j^n_0(\bar{D}^q\alpha_{n-1})(t)/t^n)-t\cdot \bar{D}_n^q\beta = 0.
\end{align*}
Hence $\alpha_n$ defines an $n$-th order extension of $\alpha$.
\end{proof}

Therefore, if $O_{n,n-1}^q\equiv 0$ for all $n\geq 1$, then by $(1)$ above we obtain a formal element $\alpha(t)$ such that $\bar{D}_t\alpha(t)=0$. In Appendix~\ref{sec:covergence}, we show that after a gauge fixing, $\alpha(t)$ is analytic in a neighborhood around $0\in B$.

\begin{remark}
The radius of convergence of each extension $\alpha(t)$ may be different as $\alpha=\alpha(0)$ varies. However, since $H^q(X,E)$ is finite dimensional, we can simply choose a basis, for instance, one consisting of harmonic forms with respect to a fixed hermitian metric. Then we obtain a minimum radius of convergence, uniform in all $[\alpha]\in H^q(X,E)$.
\end{remark}

Next we shall demonstrate that there is another obstruction for an extension to be nonzero.

\begin{proposition}
A non-exact element $\beta\in\ker(\dbar_E^q)$ admits a local extension $\beta(t)\in\Gamma(U,\pi_*\mathcal{D}^q)$ such that $\beta(t)$ is exact for $t\neq 0$ if and only if there exist $n\geq 1$ and $[j_0^{n-1}(\alpha_{n-1})(t)]\in H^{q-1}((\pi_*\mathcal{D}^{\bullet})_0\otimes\mathcal{O}_{B,0}/m_0^n)$ such that
$$O_n^{q-1}[j_0^{n-1}(\alpha_{n-1})(t)]=[\beta].$$
\end{proposition}
\begin{proof}
Suppose that $O_n^{q-1}[j_0^{n-1}(\alpha_{n-1})(t)]=[\beta]$. Then
$$\beta=j_0^n(\bar{D}^{q-1}\alpha_{n-1})(t)/t^n+\dbar^{q-1}_E\gamma$$
for some $\gamma\in\Omega^{0,q-1}(E)$. Define $\beta(t)$ by
$$\beta(t):=\bar{D}^{q-1}(\alpha_{n-1}(t)/t^n)+\bar{D}^{q-1}\gamma(t), \quad t\neq 0,$$
where $\gamma(t)$ is any extension of $\gamma$. Clearly $\beta(t)$ can be extended through the origin by setting $\beta(0)=\beta$. Then $\beta(t)$ is a $\bar{D}^{q-1}$-exact class and equals $\beta$ at $t=0$. Hence $\beta(t)$ serves as an extension of $\beta$ which is $\bar{D}^{q-1}$-exact for $t\neq 0$.

Conversely, if $\beta(t)$ is an extension of $\beta$ such that
$$\beta(t)=\bar{D}^{q-1}\gamma(t)$$
for $t\neq 0$. Then $\gamma(t)$ can be chosen to be meromorphic in $t$ with pole order $n\geq 1$ at $t=0$. Let $\alpha_{n-1}(t):=t^n\gamma(t)$. Then $\alpha_{n-1}(t)$ is holomorphic in $t$ and
$$O_n^{q-1}[j_0^{n-1}(\alpha_{n-1})(t)]=[j_0^n(\bar{D}^{q-1}(t^n\gamma(t))/t^n]=[j_0^n(t^n\beta(t))/t^n]=[\beta].$$
This completes the proof.
\end{proof}

\begin{proposition}
Let $[j_0^{n-1}(\alpha_{n-1})(t)]\in H^{q-1}((\pi_*\mathcal{D}^{\bullet})_0\otimes\mathcal{O}_{B,0}/m_0^n)$ such that $O_n^{q-1}[j_0^{n-1}(\alpha_{n-1})(t)]\neq 0$. Then there exist $n'\leq n$ and $[j_0^{n'-1}(\alpha_{n'-1})(t)]\in H^{q-1}((\pi_*\mathcal{D}^{\bullet})_0\otimes\mathcal{O}_{B,0}/m_0^{n'})$ such that
$$O_{n,n'-1}^{q-1}[j_0^{n-1}(\alpha_{n-1})(t)]=O_{n',n'-1}^{q-1}[j_0^{n'-1}(\alpha_{n'-1})(t)]\neq 0.$$
\end{proposition}
\begin{proof}
If $O_{n,n-1}^{q-1}[j_0^{n-1}(\alpha_{n-1})(t)]\neq 0$, we can simply take $n'=n$ and $\alpha_{n'-1}=\alpha_{n-1}$. Otherwise, there exists $\alpha_1'$ such that
$$\bar{D}^{q-1}_{n-1}\alpha_1'=O_{n,n-1}^{q-1}[j_0^{n-1}(\alpha_{n-1})(t)].$$
Then we have
$$O_{n-1,n-2}^{q-1}[\alpha_1']=O_{n,n-2}^{q-1}[j_0^{n-1}(\alpha_{n-1})(t)].$$
Since $O_n^{q-1}[j_0^{n-1}(\alpha_{n-1})(t)]\neq 0$, we finally arrive at some $n'$ such that
$$O_{n,n'-1}^{q-1}[j_0^{n-1}(\alpha_{n-1})(t)]=O_{n',n'-1}^{q-1}[j_0^{n'-1}(\alpha_{n'-1})(t)]\neq 0.$$
\end{proof}

These two propositions together prove the following

\begin{corollary}\label{prop:nonexact}
Every local extension of every non-exact element $\beta\in\ker(\dbar_E^q)$ is non-exact if and only if $O_{n,n-1}^{q-1}\equiv 0$ for all $n\geq 1$.
\end{corollary}
\begin{proof}
For a fixed non-exact $\beta\in\ker(\dbar_E^q)$, if any extension of $\beta$ is non-exact, then $[\beta]\notin\text{Im}(O_{n,n-1}^{q-1})$ for all $n\geq 1$. Hence $O_{n,n-1}^{q-1}\equiv 0$.

Conversely, if there is an extension of $\beta$ such that it is exact for $t\neq 0$, then there exist $n\geq 1$ and $[j_0^{n-1}(\alpha_{n-1})(t)]\in H^{q-1}((\pi_*\mathcal{D}^{\bullet})_0\otimes\mathcal{O}_{B,0}/m_0^n)$ such that
$$O_n^{q-1}[j_0^{n-1}(\alpha_{n-1})(t)]=[\beta]\neq 0.$$
But we can also choose $n'\leq n$ and $[j_0^{n'-1}(\alpha_{n'-1})(t)]\in H^{q-1}((\pi_*\mathcal{D}^{\bullet})_0\otimes\mathcal{O}_{B,0}/m_0^{n'})$ such that
$$O_{n,n'-1}^{q-1}[j_0^{n-1}(\alpha_{n-1})(t)]=O_{n',n'-1}^{q-1}[j_0^{n'-1}(\alpha_{n'-1})(t)]\neq 0.$$
This proves the corollary.
\end{proof}

\begin{lemma}
For each $q\geq 0$, $\pi_*\mathcal{D}^q$ is a flat $\mathcal{O}_{B}$-module.
\end{lemma}
\begin{proof}
This follows from the fact that $(\pi_*\mathcal{D}^q)_t$ is torsion free and $\mathcal{O}_{B,t}\cong\mathbb{C}\{x-t\}$ is a PID for every $t\in B$.
\end{proof}

We will need the following fact from homological algebra, whose proof can be found, e.g. in \cite{Hartshorne_AG}.
\begin{proposition}\label{prop:hom_alg}
Let $A$ be a Noetherian ring and $C^{\bullet}$ be a finite cochain complex of flat $A$-modules whose cohomology $H^i(C^{\bullet})$ is finitely generated for all $i$. Then there exists a cochain complex of finitely generated flat $A$-modules $K^{\bullet}$ and a cochain map $C^{\bullet}\rightarrow K^{\bullet}$, which is a quasi-isomorphism. Moreover, for any $A$-module $M$, the natural map $C^{\bullet}\otimes M\rightarrow K^{\bullet}\otimes M$ is a quasi-isomorphism. Furthermore, if the dimension
$$\dim_{k(\mathfrak{p})}H^q(K^{\bullet}\otimes k(\mathfrak{p}))$$
is locally constant in $\mathfrak{p}\in\text{Spec}(A)$, then for $i=q,q-1$, the $\delta$-functors $T^i(M):=H^i(K^{\bullet}\otimes M)$ commute with base change.
\end{proposition}

We apply this proposition to the case $A=\mathcal{O}_{B,0}$, $C^{\bullet}=(\pi_*\mathcal{D}^{\bullet})_0$ to prove the following:

\begin{proposition}\label{prop:tensor_commutes_cohomology}
If $\dim_{k(t)}H^q((\pi_*\mathcal{D}^{\bullet})_t\otimes k(t))$ is locally constant around $0\in B$, then the canonical map
$$H^q((\pi_*\mathcal{D}^{\bullet})_0)\otimes k(0)\rightarrow H^q((\pi_*\mathcal{D}^{\bullet})_0\otimes k(0))$$
is an isomorphism.
\end{proposition}
\begin{proof}
Since $(\pi_*\mathcal{D}^{\bullet})_0$ is a flat $\mathcal{O}_{B,0}$-module, using Proposition \ref{prop:hom_alg}, we obtain a complex of finitely generated flat $\mathcal{O}_{B,0}$-modules $K^{\bullet}$ such that
$$H^{\bullet}((\pi_*\mathcal{D}^{\bullet})_0\otimes M)\cong H^{\bullet}(K^{\bullet}\otimes M)$$
for any $\mathcal{O}_{B,0}$-module $M$. We claim that the dimension
$$\dim_{k(\mathfrak{p})}H^q(K^{\bullet}\otimes k(\mathfrak{p}))$$
is locally constant in $\mathfrak{p}\in\text{Spec}(\mathcal{O}_{B,0})$.

First of all, since $\dim_{k(t)}H^q((\pi_*\mathcal{D}^{\bullet})_t\otimes k(t))$ is locally constant, by Theorem \ref{prop:locally_freeness} and Proposition \ref{lem:acyclic_resolution}, the sheaf $\mathcal{H}^q(\pi_*\mathcal{D}^{\bullet})\cong R^q\pi_*\mathcal{E}$ is a locally free $\mathcal{O}_B$-module. Hence
$$H^q((\pi_*\mathcal{D}^{\bullet})_0)\otimes k(0)\cong (R^q\pi_*\mathcal{E})_0\otimes k(0)\cong H^q(X,E)\cong H^q((\pi_*\mathcal{D}^{\bullet})_0\otimes k(0)).$$
In particular
\begin{align*}
\dim_{k(0)}H^q(K^{\bullet}\otimes k(0))
& = \dim_{k(0)}H^q((\pi_*\mathcal{D}^{\bullet})_0\otimes k(0))\\
& = \dim_{k(0)}H^q((\pi_*\mathcal{D}^{\bullet})_0)\otimes k(0)\\
& = \dim_{k(0)}H^q(K^{\bullet})\otimes k(0).
\end{align*}
Note that $\text{Spec}(\mathcal{O}_{B,0})=\text{Spec}(\mathbb{C}\{x\})=\{(0),(x)\}$. Let $Q:=(\mathcal{O}_{B,0})_{(0)}$ be the localization of $\mathcal{O}_{B,0}$ at the ideal $(0)$, which is the field of quotients of $\mathcal{O}_{B,0}$. We obtain
$$H^q(K^{\bullet}\otimes k((0)))=H^q(K^{\bullet}\otimes Q)\cong H^q(K^{\bullet})\otimes Q,$$
since localization is flat. On the other hand, as $(\mathcal{O}_{B,0})_{(x)}\cong\mathcal{O}_{B,0}$, we have
$$H^q(K^{\bullet}\otimes k((x)))\cong H^q(K^{\bullet}\otimes\mathcal{O}_{B,0}/m_0)=H^q(K^{\bullet}\otimes k(0)),$$
and so
$$\dim_{k((x))}H^q(K^{\bullet}\otimes k((x)))=\dim_{k(0)}H^q(K^{\bullet}\otimes k(0))=\dim_{k(0)}H^q(K^{\bullet})\otimes k(0).$$
As $H^q(K^{\bullet})\cong H^q((\pi_*\mathcal{D}^{\bullet})_0)$ is a free $\mathcal{O}_{B,0}$-module and $\mathcal{O}_{B,0}$ is a local integral domain, we have
$$\dim_QH^q(K^{\bullet})\otimes Q=\dim_{k(0)}H^q(K^{\bullet})\otimes k(0).$$
In summary, we conclude that
$$\dim_QH^q(K^{\bullet}\otimes Q)=\dim_{k((x))}H^q(K^{\bullet}\otimes k((x))),$$
which means that $\dim_{k(\mathfrak{p})}H^q(K^{\bullet}\otimes k(\mathfrak{p}))$ is constant in $\mathfrak{p}\in\text{Spec}(\mathcal{O}_{B,0})$. Hence $T^q$ commutes with base change. The required isomorphism now follows from taking $M=k(0)$ in Proposition~\ref{prop:hom_alg}.
\end{proof}

\begin{remark}
By replacing $0\in B$ by nearby $t\in B$, we note that the isomorphism holds in a neighborhood of $0$.
\end{remark}

We are now ready to prove our main result.
\begin{theorem}\label{thm:main_2}
$\dim_{k(t)}H^q(\mathcal{X}_t,\mathcal{E}_t)$ is locally constant if and only if $O_{m,m-1}^q\equiv 0$ and $O^{q-1}_{n,n-1}\equiv 0$ for all $m,n\geq 1$.
\end{theorem}
\begin{proof}
If $\dim_{k(t)}H^q(\mathcal{X}_t,\mathcal{E}_t)=\dim_{k(t)}H^q((\pi_*\mathcal{D}^{\bullet})_t\otimes k(t))$ is locally constant, then Proposition \ref{prop:tensor_commutes_cohomology} shows that the natural map
$$H^q((\pi_*\mathcal{D}^{\bullet})_0)\otimes k(0)\rightarrow H^q((\pi_*\mathcal{D}^{\bullet})_0\otimes k(0))$$
is an isomorphism.

Now, let $c\in H^q((\pi_*\mathcal{D}^{\bullet})_0\otimes k(0))$. We can extend it to a nonzero local holomorphic section of $\mathcal{H}^q(\pi_*\mathcal{D}^{\bullet})$ since $\mathcal{H}^q(\pi_*\mathcal{D}^{\bullet})$ is locally free. Denote this extension by $\widetilde{c}$. Consider the germ of this section $\widetilde{c}_0\in H^q((\pi_*\mathcal{D}^{\bullet})_0)$. Choose a representative $\widetilde{\alpha}_0\in(\pi_*\mathcal{D}^{\bullet})_0$ in this cohomology class. For each $m\geq 1$, $\widetilde{\alpha}_0$ is mapped to $(\pi_*\mathcal{D}^{\bullet})_0\otimes\mathcal{O}_{B,0}/m_0^{m+1}$ via the quotient map $p_{m}$. Then the class $[p_{m}(\widetilde{\alpha}_0)]\in H^q((\pi_*\mathcal{D}^{\bullet})_0\otimes\mathcal{O}_{B,0}/m_0^{m+1})$ is an $m$-th order extension of $c$. Hence $O_{m,m-1}^q\equiv 0$ by Proposition \ref{prop:extension}. Since $m$ is arbitrary, $O^q_{m,m-1}\equiv 0$ for all $m\geq 1$.

For the obstruction map $O_{n,n-1}^{q-1}$, if $O_{n,n-1}^{q-1}[j_0^{n-1}(\alpha_{n-1})(t)]\neq 0$ for some $n\geq 1$ and $[j_0^{n-1}(\alpha_{n-1})(t)]\in H^q((\pi_*\mathcal{D}^{\bullet})_0\otimes\mathcal{O}_{B,0}/m_0^n)$, then we can find some nonzero $[\beta]\in H^q(\pi_*\mathcal{D}^{\bullet}_0\otimes k(0))$ and a local holomorphic extension $\widetilde{\beta}$ of $\beta$ such that it is exact only when $t\neq 0$. But since $\mathcal{H}^q(\pi_*\mathcal{D}^{\bullet})$ is locally free, any extension is locally nonzero by continuity. Therefore, $O_{n,n-1}^{q-1}\equiv 0$ for all $n\geq 1$ by Proposition \ref{prop:nonexact}.

Conversely, if both obstruction maps vanish, then for each $[\alpha]\in H^q((\pi_*\mathcal{D}^{\bullet})_0\otimes k(0))$, we obtain $\alpha(t)\in\Gamma(U,\pi_*\mathcal{D}^q)$ such that $\bar{D}_t\alpha(t)=0$ in some neighborhood $U\subset B$ containing $0$ and $[\alpha(0)]=[\alpha]$. Moreover, $\alpha(t)$ is non-exact since $O_{n,n-1}^{q-1}\equiv 0$ for all $n\geq 1$. Hence for fixed $t\in U$ we obtain an injective linear map
$$H^q((\pi_*\mathcal{D}^{\bullet})_0\otimes k(0))\rightarrow H^q((\pi_*\mathcal{D}^{\bullet})_t\otimes k(t)),\quad [\alpha]\mapsto[\alpha(t)].$$
Therefore,
$$\dim_{k(0)}H^q((\pi_*\mathcal{D}^{\bullet})_0\otimes k(0))\leq\dim_{k(t)}H^q((\pi_*\mathcal{D}^{\bullet})_t\otimes k(t)).$$
By upper semi-continuity, $\dim_{k(t)}H^q((\pi_*\mathcal{D}^{\bullet})_t\otimes k(t))=\dim_{k(t)}H^q(\mathcal{X}_t,\mathcal{E}_t)$ is locally constant.
\end{proof}

Recall that by choosing a Hermitian metric on $E$ and using the associated Chern connection, we can write
$$\bar{D}_t=\dbar+\varphi(t)\lrcorner\nabla+A(t),$$
where $\{(A(t),\varphi(t))\}_{t\in B}$ is the family of Maurer-Cartan elements which controls the deformations of $(X,E)$. Hence the $n$-th order obstruction maps $O^{i}_{n,n-1}:H^i((\pi_*\mathcal{D}^{\bullet})_0\otimes\mathcal{O}_{B,0}/m_0^n)\rightarrow H^{i+1}((\pi_*\mathcal{D}^{\bullet})_0\otimes\mathcal{O}_{B,0}/m_0^n)$, for $i=q,q-1$, defined in \eqref{eqn:obstruction} can be rewritten as
\begin{equation}\label{eqn:obstruction2}
O^{i}_{n,n-1}\left([\alpha_{n-1}]\right)=\left[t^{n-1}\sum_{j=0}^{n-1}(\varphi^{n-j}\lrcorner\nabla + A^{n-j})\alpha^j_{n-1}\right].
\end{equation}
as claimed in Theorem \ref{thm:main}.

\begin{example}
We first consider the case when $E=T_X$, the holomorphic tangent bundle of $X$. We deform the pair $(X,T_X)$ to $(\cu{X}_t,T_{\cu{X}_t})$, where $T_{\cu{X}_t}$ is the holomorphic tangent bundle to $\cu{X}_t$ (note that $T_X$ may have other deformations which are not isomorphic to the holomorphic tangent bundle on $\cu{X}_t$). In this case, the $\text{End}(T_X)$-part of the Maurer-Cartan element $(A(t),\varphi(t))$ is given by
$$A(t)=-T(\varphi(t),\bullet)-\nabla_{\bullet}\varphi(t),$$
where $T:\Omega^{0,\bullet}(T_X)\times\Omega^{0,\bullet}(T_X)\rightarrow\Omega^{0,\bullet}(T_X)$ is the graded torsion on $T_X$ defined by
$$T(\varphi,\psi):=\varphi\lrcorner\nabla\psi-(-1)^{|\varphi||\psi|}\psi\lrcorner\nabla\varphi-[\varphi,\psi].$$
So we have
$$\bar{D}_t^{\bullet}=\dbar_{T_X}^{\bullet}+[\varphi(t),-].$$

For $\alpha_{n-1}\in\Omega^{0,q}(T_X)\otimes\mathcal{O}_{B,0}$ such that $\bar{D}_t^q\alpha_{n-1}=0$ mod $t^{n-1}$, we have
$$t^{n-1}(j^n_0(\bar{D}^q_t\alpha_{n-1})/t^n)=t^{n-1}\left(\dbar_{T_X}\alpha^n_{n-1}+\sum_{j=0}^{n-1}[\varphi^{n-j},\alpha_{n-1}^j]\right).$$
As a class in $H^{q+1}(\Omega^{0,\bullet}(T_X)\otimes\mathcal{O}_{B,0}/m_0^n,\bar{D}_{n-1}^{\bullet})$, it is equal to
$$\left[t^{n-1}\sum_{j=0}^{n-1}[\varphi^{n-j},\alpha_{n-1}^j]\right].$$
Hence the obstruction is given by
$$O^q_{n,n-1}[j_0^{n-1}(\alpha_{n-1})(t)]=\left[t^{n-1}\sum_{j=0}^{n-1}[\varphi^{n-j},\alpha_{n-1}^j]\right].$$
\end{example}

\begin{example}
For the case $E=T_X^*$, we have
$$\bar{D}_t^{\bullet}=\dbar_{T_X^*}^{\bullet}+[\varphi(t),-]^*,$$
where $[\varphi(t),-]^*:\Omega^{0,q}(T_X^*)\rightarrow\Omega^{0,q+1}(T_X^*)$ is given by $$[\varphi(t),\eta]^*(v):=[\varphi(t),\eta(v)]-(-1)^q\eta([\varphi(t),v])=\varphi(t)\lrcorner\partial(\eta(v))-(-1)^q\eta([\varphi(t),v])$$
for $v\in\Omega^0(T_X)$. Since $$\varphi(t)\lrcorner\partial(\eta(v))-(-1)^q\eta([\varphi(t),v])=(\varphi(t)\lrcorner\partial\eta)(v)+v\lrcorner\partial(\varphi(t)\lrcorner\eta),$$ the obstruction is given by
$$O^q_{n,n-1}[j_0^{n-1}(\alpha_{n-1})(t)]=\left[t^{n-1}\sum_{j=0}^{n-1}(\varphi^{n-j}\lrcorner\partial\alpha_{n-1}^j + \partial(\varphi^{n-j}\lrcorner\alpha_{n-1}^j))\right].$$

For $E=\wedge^qT^*_X$, we have
\begin{align*}
\bar{D}_t(\alpha_1\wedge\cdots\wedge\alpha_p) & = \sum_{j=1}^{p-1}(-1)^{j-1}\alpha_1\wedge\cdots\wedge\bar{D}_t\alpha_j\wedge\cdots\wedge\alpha_p\\
& = \dbar(\alpha_1\wedge\cdots\wedge\alpha_p)+\sum_{j=1}^p(-1)^{j-1}\alpha_1\wedge\cdots\wedge[\varphi(t),\alpha_j]^*\wedge\cdots\wedge\alpha_p,
\end{align*}
where $\alpha_j\in\Omega^0(T^*_X)$. Then
\begin{align*}
  & \sum_{j=1}^{p-1}(-1)^{j-1}\alpha_1\wedge\cdots\wedge(\varphi(t)\lrcorner\partial\alpha_j)\wedge\cdots\wedge\alpha_p + \sum_{j=1}^{p-1}(-1)^{j-1}\alpha_1\wedge\cdots\wedge\partial(\varphi(t)\lrcorner\alpha_j)\wedge\cdots\wedge\alpha_p\\
= & \varphi(t)\lrcorner(\partial(\alpha_1\wedge\cdots\wedge\alpha_p))+\partial(\varphi(t)\lrcorner(\alpha_1\wedge\cdots\wedge\alpha_p)).
\end{align*}
Hence the obstruction map is given by
$$O^q_{n,n-1}[j_0^{n-1}(\alpha_{n-1})(t)] = \left[t^{n-1}\sum_{j=0}^{n-1}(\varphi^{n-j}\lrcorner\partial\alpha_{n-1}^j + \partial(\varphi^{n-j}\lrcorner\alpha_{n-1}^j))\right].$$
\end{example}

These two examples recover the obstruction formulae in \cite{Ye_jumping, Ye_jumping_tangent}.

\section{An application: jumping of $\dim_{\bb{C}}H^1(\cu{X}_t,\text{End}(T_{\cu{X}_t}))$}\label{sec:CY_stable}

Physicists are interested in knowing whether the dimension of the cohomology group $H^1(\cu{X}_t,\text{End}(T_{\cu{X}_t}))$ is locally constant under small deformations $X$ \cite{Huybrechts95}. The expectation is that $\dim_\mathbb{C} H^1(\cu{X}_t,\text{End}(T_{\cu{X}_t}))$ does not jump along {\em any} deformation of a Calabi-Yau manifold $X$.

In this section, we apply our results to prove a weaker statement, namely, the constancy of this dimension when the Calabi-Yau manifold $X$ satisfies an extra unobstructedness assumption.
We will first prove that, in some nice (but restrictive) cases, the dimension $\dim_{\mathbb{C}}H^1(\cu{X}_t, A(\cu{E}_t))$ does not jump at $t=0$ for any deformation of the pair $(X,E)$.

To do this, we choose a harmonic basis $\{(A_i,\varphi_i)\}_{i=1}^m$ for $H^1(X,A(E))$. In \cite{CS_pair}, we proved that the obstruction map $Ob_{(X,E)}:H^1(X,A(E))\rightarrow H^2(X,A(E))$ of the deformation theory of $(X,E)$ is given by
$$Ob_{(X,E)}:\dpsum{i=1}{m}t_i(A_i,\varphi_i)\mapsto\bb{H}[(A(t),\varphi(t)),(A(t),\varphi(t))],$$
where $\bb{H}$ is the harmonic projection and $(A(t),\varphi(t))$ satisfies
$$(A(t),\varphi(t))=\dpsum{i=1}{m}t_i(A_i,\varphi_i)-\frac{1}{2}\dbar^*_{A(E)}G_E[(A(t),\varphi(t)),(A(t),\varphi(t))],$$
where $G_E$ is the Green's operator and $\dbar^*_{A(E)}$ is  the formal adjoint of $\dbar_{A(E)}$. Moreover, $(A(t),\varphi(t))$ satisfies the Maurer-Cartan equation if and only if $Ob_{(X,E)}=0$.

Suppose now $Ob_{(X,E)}=0$. Then we have
$$\dbar_{A(E)}(A(t),\varphi(t))+\frac{1}{2}[(A(t),\varphi(t)),(A(t),\varphi(t))]=0.$$
Differentiating $(A(t),\varphi(t))$ with respect to $t_i$ and setting $t=0$, we get
$$\pd{}{t_i}|_{t=0}(A(t),\varphi(t))=(A_i,\varphi_i).$$
Hence, for each $i=1,\dots,m$, if we define $(B(t),\psi(t))_i$ to be
$$(B(t),\psi(t))_i=\pd{}{t_i}(A(t),\varphi(t)),$$
then $(B(t),\psi(t))_i$ satisfies
$$\dbar_{A(E)}(B(t),\psi(t))_i+[(A(t),\varphi(t)),(B(t),\psi(t))_i]=0$$
and $\{(B_0,\psi_0)_i\}_{i=1}^m$ forms a basis for $H^1(X,A(E))$.

Note that the differential operator $\bar{D}_{A(\cu{E}_t)}$ defined by
$$\bar{D}_{A(\cu{E}_t)}:=\dbar_{A(E)}+[(A(t),\varphi(t)),-]$$
satisfies $\bar{D}_{A(\cu{E}_t)}^2=0$ and the Leibniz rule
$$\bar{D}_{A(\cu{E}_t)}(fs)=(\dbar+\varphi(t)\lrcorner\partial)f\otimes s+f\bar{D}_{A(\cu{E}_t)}s.$$
It follows that $\bar{D}_{A(\cu{E}_t)}$ defines a deformation $\{(\cu{X}_t,A(\cu{E}_t))\}_{t\in \text{Def}(X,E)}$ of the pair $(X,A(E))$. In fact, $A(\cu{E}_t)$ is the Atiyah extension of the deformed bundle $\cu{E}_t$ on $\cu{X}_t$.

\begin{lemma}\label{lemma:unobstructed_extend}
Suppose $Ob_{(X,E)}=0$. Then for any $[(B,\psi)]\in H^1(X,A(E))$, there exists $(B(t),\psi(t))$ such that $\bar{D}_{A(\cu{E}_t)}(B(t),\psi(t))=0$ and
$[(B_0,\psi_0)]=[(B,\psi)]$. Hence any element in $H^1(X,A(E))$ admits an extension to $H^1(\cu{X}_t,A(\cu{E}_t))$ for any deformation of $(X,E)$. In particular $O_{n,n-1}^1\equiv 0$ for all $n\geq 1$.
\end{lemma}
\begin{proof}
Since this is true for the harmonic basis $\{(B_0,\psi_0)_i\}_{i=1}^m$, it is true for any element in $H^1(X,A(E))$.
\end{proof}

\begin{lemma}\label{lemma:unobstructed_jumping}
Let $X$ be a compact complex manifold and $E\rightarrow X$ be a holomorphic vector bundle. Suppose the deformation of the pair $(X,E)$ is always unobstructed and $\dim_{\mathbb{C}}H^0(\cu{X}_t,A(\cu{E}_t))$ does not jump at $t=0$ along any deformations of $(X,E)$. Then $\dim_{\mathbb{C}}H^1(\cu{X}_t,A(\cu{E}_t))$ does not jump at $t=0$ along any deformation of $(X,E)$.
\end{lemma}
\begin{proof}
Since $Ob_{(X,E)}=0$, Lemma \ref{lemma:unobstructed_extend} allows us to extend any element in $H^1(X,A(E))$ to $H^1(\cu{X}_t,A(\cu{E}_t))$. Since
\begin{align*}
H^0(\cu{X}_t,A(\cu{E}_t))& = \ker(\dbar_{A(\cu{E}_t)}:\Omega^0(A(E))\rightarrow\Omega^{0,1}_t(A(E)))\\
& = \ker(\bar{D}_{A(\cu{E}_t)}:\Omega^0(A(E))\rightarrow\Omega^{0,1}(A(E))),
\end{align*}
the assumption that $\dim_{\mathbb{C}}H^0(\cu{X}_t,A(\cu{E}_t))$ does not jump at $t=0$ implies $O_{n,n-1}^0\equiv 0$ for all $n\geq 1$. Now apply Theorem \ref{thm:main_2}.
\end{proof}

We are now going to prove that under certain assumptions, $\dim_{\mathbb{C}}H^1(\cu{X}_t,\text{End}(T_{\cu{X}_t}))$ does not jump at $t=0$ along any deformation of $X_t$.

First, when $E=T_X$, we have a canonical lift $L:H^1(X,T_X)\rightarrow H^1(X,A(T_X))$, defined by
$$L:\varphi\mapsto (-\nabla_{\bullet}\varphi-T(\varphi,\bullet),\varphi),$$
where $T:\Omega^{0,p}(T_X)\otimes\Omega^{0,q}(T_X)\rightarrow\Omega^{0,p+q}(T_X)$ is the graded torsion, defined by
$$T(\varphi,\psi)=\varphi\lrcorner\nabla\psi-(-1)^{pq}\psi\lrcorner\nabla\varphi-[\varphi,\psi].$$
Moreover, if $Ob_X=0$, then we have a Maurer-Cartan element $\varphi(t)\in\Omega^{0,1}(T_X)$ and we obtain a deformation of $(X,T_X)$ by
$$\bar{D}_t=\dbar_{T_X}+\varphi(t)\lrcorner\nabla-\nabla_{\bullet}\varphi(t)-T(\varphi(t),\bullet)=\dbar_{T_X}+[\varphi(t),\bullet].$$
In fact, the deformation induced by this operator is isomorphic to the family $\{(\cu{X}_t,T_{\cu{X}_t})\}_{t\in \text{Def}(X)}$, where $T_{\cu{X}_t}$ is the holomorphic tangent bundle of $\cu{X}_t$. Therefore, $L$ induces a natural embedding
$$\text{Def}(X) \subset \text{Def}(X,T_X).$$

By a {\em Calabi-Yau} $n$-fold we mean an $n$-dimensional compact K\"ahler manifold $X$ with trivial canonical line bundle $K_X \cong \mathcal{O}_X$ and also $H^{0,p}(X)=0$ for all $p\neq 0,n$.

\begin{theorem}\label{thm:CY_jumping}
Suppose that $X$ is a Calabi-Yau manifold such that deformations of the pair $(X,T_X)$ are unobstructed. Then $\dim_{\mathbb{C}}H^1(\cu{X}_t,\text{End}(T_{\cu{X}_t}))$ does not jump at $t=0$ for any deformation of $X$ .
\end{theorem}
\begin{proof}
Let $E=T_X$. Since the pair $(X,E)$ admits unobstructed deformations, Lemma \ref{lemma:unobstructed_extend} allows us to extend any element in $H^1(X,A(E))$ to element in $H^1(\cu{X}_t,A(\cu{E}_t))$, where $t\in \text{Def}(X,E)$. Consider the Atiyah exact sequence of $\cu{E}_t$ (note: this may {\em not} be the tangent bundle of $\cu{X}_t$ in general!) over $\cu{X}_t$:
$$0\rightarrow\text{End}(\cu{E}_t)\rightarrow A(\cu{E}_t)\rightarrow T_{\cu{X}_t}\rightarrow 0,$$
which gives rise to the injective map $\iota^*_t:H^0(\cu{X}_t,\text{End}(\cu{E}_t))\rightarrow H^0(\cu{X}_t, A(\cu{E}_t))$. Since the tangent bundle of a Calabi-Yau manifold is stable, we have $H^0(X,\text{End}_0(T_X))=0$ and so
$$H^0(X,A(E))\cong H^0(X,\text{End}(T_X))\cong H^0(X,\mathcal{O}_X)=\mathbb{C}.$$
Since the identity map $id_{\cu{E}_t}$ is always a non-zero holomorphic section of $H^0(\cu{X}_t,\text{End}(\cu{E}_t))$ and $\iota^*_t:H^0(\cu{X}_t,\text{End}(\cu{E}_t)\rightarrow H^0(\cu{X}_t,A(\cu{E}_t))$ is injective, we get
$$1 \leq \dim_{\mathbb{C}}H^0(\cu{X}_t,A(\cu{E}_t)) \leq \dim_{\mathbb{C}}H^0(X,A(E)) = 1$$
for $|t|$ small. By Lemma \ref{lemma:unobstructed_jumping}, we conclude that $\dim_{\mathbb{C}}H^1(\cu{X}_t,A(\cu{E}_t))$ does not jump at $t=0$ along any deformation of $(X,T_X)$, and in particular, along any deformation of $X$ itself.

For $t\in \text{Def}(X)\subset \text{Def}(X,T_X)$, we have a family of canonical lifts $L_t:H^1(\cu{X}_t,T_{\cu{X}_t})\rightarrow H^1(\cu{X}_t,A(\cu{E}_t))$, since $A(\cu{E}_t)$ is the Atiyah extension of $T_{\cu{X}_t}$ for $t\in \text{Def}(X)$. So the map $\pi^*_t:H^1(\cu{X}_t,A(\cu{E}_t))\rightarrow H^1(\cu{X}_t,T_{\cu{X}_t})$ is surjective and we obtain the following exact sequence
$$0\rightarrow H^1(\cu{X}_t,\text{End}(T_{\cu{X}_t}))\rightarrow H^1(\cu{X}_t,A(\cu{E}_t))\rightarrow H^1(\cu{X}_t,T_{\cu{X}_t})\rightarrow 0.$$
Since $\dim_{\mathbb{C}}H^1(\cu{X}_t,T_{\cu{X}_t})=\dim_{\mathbb{C}}H^{n-1,1}(\cu{X}_t)$ does not jump at $t=0$ for $t\in \text{Def}(X)$ with $|t|$ small, we see that $\dim_{\mathbb{C}}H^1(\cu{X}_t,\text{End}(T_{\cu{X}_t}))$ does not jump at $t=0$ for any deformation of $X$.
\end{proof}

\appendix

\section{Convergence}\label{sec:covergence}

Consider an element $\alpha\in\ker(\dbar^q)$. Suppose that the obstruction maps $O_{n,n-1}^q$ vanish for all $n\geq 1$. Then we obtain a formal extension $\alpha(t)$ of $\alpha$, that is, as a formal power series in $\Omega^{0,q}(E)$,
$$\bar{D}_t^q\alpha(t)=0.$$
In this appendix, we show that one can always choose an extension $\alpha(t)$ with a nonzero radius of convergence. To achieve this, we shall work on the Kuranishi family of $(X,E)$ \cite{Kuranishi65}, following the approach of the book \cite{Morrow-Kodaira_book}.

We choose a hermitian metric for $E$ and consider the equation
$$\alpha(t)+\dbar^*_EG_E(\varphi(t)\lrcorner\nabla+A(t))\alpha(t)=0,\quad \alpha(0)=\alpha\in\ker(\dbar^q_E),$$
with $\alpha(t)$ holomorphic in the variable $t$. Then $\alpha(t)$ can be solved by the recursive relations:
$$\alpha^n+\sum_{i=0}^{n-1}\dbar^*_EG_E(\varphi_{n-i}\lrcorner\nabla+A_{n-i})\alpha^i=0,\quad n\geq 1.$$
We shall prove that $\alpha(t):=\sum_{n=0}^{\infty}\alpha^nt^n$ converges uniformly in the H\"older norm $\|\cdot\|_{k+\alpha}$. First of all, let us recall the obvious estimates
\begin{align*}
\|[(A,\varphi),(B,\psi)]\|_{k+\alpha} & \leq C_{k,\alpha}\|(A,\varphi)\|_{k+\alpha+1}\|(B,\psi)\|_{k+\alpha+1},\\
\|(\varphi\lrcorner\nabla+A)\delta\|_{k+\alpha} & \leq C_{k,\alpha}'\|(A,\varphi)\|_{k+\alpha+1}\|\delta\|_{k+\alpha+1}
\end{align*}
for any $(A,\varphi),(B,\psi)\in\Omega^{\bullet}(\mathcal{E})$ and $\delta\in\Omega^{0,\bullet}(E)$, where $C_{k,\alpha},C_{k\alpha}'$ are positive constants which depend only on $k,\alpha$. We may assume that $C_{k,\alpha}$ is larger so that
\begin{align*}
\|[(A,\varphi),(B,\psi)]\|_{k+\alpha} & \leq C_{k,\alpha}\|(A,\varphi)\|_{k+\alpha}\|(B,\psi)\|_{k+\alpha},\\
\|(\varphi\lrcorner\nabla+A)\delta\|_{k+\alpha} & \leq C_{k,\alpha}\|(A,\varphi)\|_{k+\alpha}\|\delta\|_{k+\alpha}
\end{align*}
for any $(A,\varphi),(B,\psi)\in\Omega^{\bullet}(A(E))$ and $\delta\in\Omega^{0,\bullet}(E)$.
Next, we have the estimates
\begin{align*}
\|\dbar_E^*G_E\delta\|_{k+\alpha} & \leq \widetilde{C}_{k,\alpha}\|\delta\|_{k-1+\alpha},\\
\|\dbar_{A(E)}^*G_{A(E)}(A,\varphi)\|_{k+\alpha} & \leq \widetilde{C}'_{k,\alpha}\|(A,\varphi)\|_{k-1+\alpha}
\end{align*}
for all $(A,\varphi)\in\Omega^{0,\bullet}(A(E))$ and $\delta\in\Omega^{0,\bullet}(E)$, where $G_{A(E)},G_E$ are Green's operators correspond to $A(E),E$, respectively, and $\widetilde{C}_{k,\alpha},\widetilde{C}'_{k,\alpha}$ are positive constants depending only on $k,\alpha$. Again we assume that $\widetilde{C}_{k,\alpha}$ is larger.

\begin{proposition}
For $|t|$ small, $\alpha(t)=\sum_{n=0}^{\infty}\alpha^nt^n$ converges in the norm $\|\cdot\|_{k+\alpha}$ and $\alpha(t)$ is a smooth solution.
\end{proposition}
\begin{proof}
The proof is rather standard, and we follow the book \cite{Morrow-Kodaira_book} very closely.

First we observe that $\delta(t):=t\cdot\alpha(t)$ also satisfies the equation
$$\delta(t)+\dbar^*_EG_E((\varphi(t)\lrcorner\nabla+A(t))\delta(t))=0.$$
Denote $\delta_n(t)=\delta(t)\text{ mod }t^{n+1}$ (similar meaning for $A^n(t)$ and $\varphi^n(t)$). Let
$$B(t):=\frac{\beta}{16\gamma}\sum_{n=1}^{\infty}\frac{\gamma^n}{n^2}t^n:=\sum_{n=1}^{\infty}B^nt^n,$$
where $\beta,\gamma$ are positive constants which are to be chosen. We want to choose $\beta,\gamma$ such that $\|\delta^n\|_{k+\alpha}\leq B^n$ for all $n\geq 1$ (this condition will be denoted by $\|\delta_n(t)\|_{k+\alpha} \ll B(t)$). This is of course possible for $n=1$. Hence we assume that this is possible up to order $n-1$, for some $n>1$.

For any $(A,\varphi)$ and $\delta$, we have
$$\|\dbar_E^*G_E((\varphi\lrcorner\nabla+A)\delta)\|_{k+\alpha}\leq \widetilde{C}_{k,\alpha}\|(\varphi\lrcorner\nabla+A)\delta\|_{k-1+\alpha}\leq
\widetilde{C}_{k,\alpha}C_{k,\alpha}\|(A,\varphi)\|_{k+\alpha}\|\delta\|_{k+\alpha},$$
so the induction hypothesis gives
\begin{align*}
\|\delta_n(t)\|_{k+\alpha} & \leq \widetilde{C}_{k,\alpha}C_{k,\alpha}\|(A^n(t),\varphi^n(t))\|_{k+\alpha}\|\delta_{n-1}(t)\|_{k+\alpha}\\
& \ll \widetilde{C}_{k,\alpha}C_{k,\alpha}\|(A^n(t),\varphi^n(t))\|_{k+\alpha}B(t).
\end{align*}
It follows from Proposition 2.4, p.162 in \cite{Morrow-Kodaira_book} that, when $\beta,\gamma$ are chosen such that
$$\widetilde{C}_{k,\alpha}C_{k,\alpha}\frac{\beta}{\gamma}<1\quad\text{ and }\quad \|(A^1(t),\varphi^1(t))\|_{k+\alpha} \ll B(t),$$
we have $\|(A^n(t),\varphi^n(t))\|_{k+\alpha} \ll B(t)$ for any $n\geq 1$. Hence
$$\|\delta_n(t)\|_{k+\alpha} \ll \widetilde{C}_{k,\alpha}C_{k,\alpha}(B(t))^2.$$
It can also be proved (see Lemma 3.6, p. 50 in \cite{Morrow-Kodaira_book}) that
$$(B(t))^2 \ll \frac{\beta}{\gamma}B(t).$$
Therefore, for the above choices of $\beta,\gamma$, we have
$$\|\delta_n(t)\|_{k+\alpha} \ll B(t).$$
Since $B(t)$ converges on $|t|<\gamma^{-1}$, we see that $\delta(t)$, and hence $\alpha(t)$, also converges there.

Finally, $\alpha(t)$ satisfies
$$\left(\frac{\partial^2}{\partial t\partial\bar{t}}+\Delta_E+\dbar_E^*(\varphi(t)\lrcorner\nabla+A(t))\right)\alpha(t) = 0.$$
Since the operator
$$\frac{\partial^2}{\partial t\partial\bar{t}}+\Delta_E+\dbar_E^*(\varphi(t)\lrcorner\nabla+A(t))$$
is elliptic for $|t|$ small, regularity guarantees smoothness of $\alpha(t)$.
\end{proof}

Next we have the following

\begin{proposition}
The $\alpha(t)$ defined above satisfies
$$\bar{D}_t\alpha(t)=(\dbar_E+\varphi(t)\lrcorner\nabla+A(t))\alpha(t)=0\text{ mod }t^n$$
if and only if $\mathbb{H}((\varphi(t)\lrcorner\nabla+A(t))\alpha(t))=0\text{ mod }t^n$.
\end{proposition}
\begin{proof}
If $\bar{D}_t\alpha(t)=0\text{ mod }t^n$, then it is clear that $\mathbb{H}((\varphi(t)\lrcorner\nabla+A(t))\alpha(t))=0\text{ mod }t^n$ since $\mathbb{H}\dbar_E=0$.

Conversely, suppose that $\mathbb{H}((\varphi(t)\lrcorner\nabla+A(t))\alpha(t))=0\text{ mod }t^n$. Let $\psi(t):=\bar{D}_t\alpha(t)$. Since $\alpha(t)$ satisfies
$$\alpha(t)+\dbar_E^*G_E(\varphi(t)\lrcorner\nabla+A(t))\alpha(t)=0,$$
applying $\dbar_E$ gives
$$\dbar_E\alpha(t)=-\dbar_E\dbar_E^*G_E(\varphi(t)\lrcorner\nabla+A(t))\alpha(t).$$
Then
$$\psi(t)=-\dbar_E\dbar^*_EG_E(\varphi(t)\lrcorner\nabla+A(t))\alpha(t)+(\varphi(t)\lrcorner\nabla+A(t))\alpha(t).$$

Since $(\varphi(t)\lrcorner\nabla+A(t))\alpha(t)\text{ mod }t^n$ has no harmonic part, we have
\begin{align*}
\psi(t) & = \dbar^*_E\dbar_EG_E(\varphi(t)\lrcorner\nabla+A(t))\alpha(t)\\
& = \dbar^*_EG_E[(\dbar_{T_X}\varphi(t)\lrcorner\nabla+\varphi(t)\lrcorner F_{\nabla}+\dbar_{\text{End}(E)}A(t))\alpha(t)\\
& \quad\quad -(\varphi(t)\lrcorner\nabla+A(t))\dbar_E\alpha(t)]\\
& = \dbar^*_EG_E\Big[-\frac{1}{2}[(A(t),\varphi(t)),(A(t),\varphi(t))]\cdot\alpha(t)\\
& \quad\quad  - (\varphi(t)\lrcorner\nabla+A(t))(\psi(t)-(\varphi(t)\lrcorner\nabla+A(t))\alpha(t))\Big]\\
& = -\dbar^*_EG_E[(\varphi(t)\lrcorner\nabla+A(t))\psi(t)]\text{ mod }t^n,
\end{align*}
where the Lie bracket acts by
$$[(A(t),\varphi(t)),(A(t),\varphi(t))]\cdot\alpha(t):=(2\varphi(t)\lrcorner\nabla+[A(t),A(t)]+[\varphi(t),\varphi(t)]\lrcorner\nabla)\alpha(t).$$
Since the leading order term of $(A(t),\varphi(t))$ is at least 1, the leading order of $\dbar^*_EG_E[(\varphi(t)\lrcorner\nabla+A(t))\psi(t)]$ is of order at least 2. Hence $\psi(t)$ has no first order term. Inductively, we conclude $\psi(t)=0\text{ mod }t^n$.

\end{proof}

Finally, we claim that the harmonic part of $(\varphi(t)\lrcorner\nabla+A(t))\alpha(t)$ vanishes under the assumption that $O_{n,n-1}^q\equiv 0$ for all $n\geq 1$.

\begin{proposition}
The obstructions $O_{n,n-1}^q\equiv 0$ for all $n\geq 1$ if and only if for any $\alpha(t)$ satisfying
$$\alpha(t)+\dbar_E^*G_E(\varphi(t)\lrcorner\nabla+A(t))\alpha(t)=0$$
and $\dbar_E\alpha(0)=0$, we have $\bar{D}_t\alpha(t)=0$.
\end{proposition}
\begin{proof}
If $\mathbb{H}((\varphi(t)\lrcorner\nabla+A(t))\alpha(t))=0$ for any $\alpha=\alpha(0)\in\ker(\dbar_E^q)$, then $\alpha(t)$ is an extension of $\alpha$. Hence $O_{n,n-1}^q\equiv 0$ for all $n\geq 1$.

For the converse direction, we proceed by induction on $n$. For $n=1$, we have
$$j_0^1(\bar{D}_t\alpha_0)(t)=\bar{D}_0(j_0^0(\beta)(t))=j_0^0(\bar{D}_t\beta)(t)$$
for some local section $\beta=\sum_{n=0}^{\infty}\beta_nt^n$, i.e.
$$(\varphi_1\lrcorner\nabla+A_1)\alpha_0=\dbar_E\beta_0.$$
Hence $\mathbb{H}((\varphi(t)\lrcorner\nabla+A(t))\alpha(t))=0\text{ mod }t$. Assume $\mathbb{H}((\varphi(t)\lrcorner\nabla+A(t))\alpha(t))=0\text{ mod }t^{n-1}$. Then $\alpha(t)$ is an $(n-1)$-th order extension of $\alpha_0$. By assumption, we have $O_{n,n-1}[j_0^{n-1}(\alpha(t))]=0$. Therefore,
$$t^{n-1}\sum_{j=0}^{n-1}(\varphi_{n-j}\lrcorner\nabla+A_{n-j})\alpha^j=\bar{D}_{n-1}(j^{n-1}_0(\beta)(t))=j^{n-1}_0(\bar{D}_t\beta)(t).$$
Hence
$$\sum_{j=0}^{n-1}(\varphi_{n-j}\lrcorner\nabla+A_{n-j})\alpha^j
=\dbar_E\beta^{n-1}+\sum_{j=0}^{n-2}(\varphi_{n-1-j}\lrcorner\nabla+A_{n-1-j})\beta^j$$
and
$$\dbar_E\beta^{k}+\sum_{j=0}^{k-1}(\varphi_{k-j}\lrcorner\nabla+A_{k-j})\beta^j=0$$
for $k \leq n-2$.

The last $(n-2)$ equations simply mean that $\beta$ defines an extension of $\beta_0$ of order $n-2$. By assumption, we have $O_{n-1,n-2}^q[j_0^{n-2}(\beta)(t)]=0$, and so
$$\sum_{j=0}^{n-2}(\varphi_{n-1-j}\lrcorner\nabla+A_{n-1-j})\beta^j=\dbar_E\gamma^{n-2}+\sum_{j=0}^{n-3}(\varphi_{n-2-j}\lrcorner\nabla+A_{n-2-j})\gamma^j$$
for some $\gamma=\sum_{n=0}^{\infty}\gamma^nt^n$. Repeating the previous argument, this reduces to the $n=1$ case, and so
$$\sum_{j=0}^{n-1}(\varphi_{n-j}\lrcorner\nabla+A_{n-j})\alpha^j$$
is $\dbar_E$-exact, and therefore, has no harmonic part. This completes the induction argument.
\end{proof}

\bibliographystyle{amsplain}
\bibliography{geometry}

\end{document}